\documentclass[smallextended]{amsart}       
\usepackage{graphicx,hyperref,colortbl,amssymb,adjustbox}
\newtheorem{theorem}{Theorem}

\newtheorem{proposition}{Proposition}
\newtheorem{lemma}{Lemma}

\begin{document}

\title{Corner the Empress}


\author{Robbert Fokkink, Gerard Francis Ortega, \and Dan Rust} 
\dedicatory{Dedicated to the memory of
Elizabeth Alexandra Mary Windsor, Elizabeth II, by the Grace of God, of the United Kingdom of Great Britain and Northern Ireland and other realms and territories. Queen, Head of the Commonwealth, Defender of the Faith, etc, etc, etc.
}

	\begin{abstract}
Wythoff Nim aka Corner the Lady is a classic 
combinatorial game.
A Queen is placed on an infinite 
chess board and two players
take alternate turns, moving the Queen closer
to the corner. The first player that corners
the Queen wins.
What happens if the Queen gets superior
powers and is able to step off the diagonal or bounce against a side?
In this paper we study the intriguing patterns that emerge
from such games.
In particular, we are interested in games in which
the $P$-positions can be described by morphic sequences.
We use the theorem-prover \texttt{Walnut} to
prove some of our results.
\end{abstract}
\maketitle

\section{All rise}
\label{intro}

The Queen is the most powerful piece on the chess board. She can move
diagonally, horizontally, and vertically for as far as it may please her, 
in
one go. It was not always so. Originally, she was a weak piece that
could only take one diagonal step at a time. In the original Indian
game Chaturanga, from which all modern versions of chess have been derived,
the Queen was known as the Mantri, or Counselor.
She still goes by that name in Xiangqi, the Chinese version of chess,
in which she is not allowed to leave the palace.
During the late Middle Ages, the Counselor changed gender and rose to 
prominence in European
chess. The modern movement of the Queen started in Spain under the
reign of Isabella I of Castile, perhaps inspired by her great political
power~\cite{Y}. In this paper, we will consider Queens with even
greater mobility.

An ordinary chess board has only 64 squares, 
which are described by letters for columns and numbers for rows in the
European notation
(of course there also exists an English notation, but no one really understands why).
Corner the Lady, or Wythoff Nim, is played on an infinite chess board 
on which the sun never sets (at least in the positive quadrant).
Its rows and columns are both numbered $0, 1, 2,\ldots$ for
lack of letters. 
Anne places the Queen on any square $(m,n)$ and Beatrix moves first.
The players
make alternate moves but only so that the Manhattan distance $m+n$
to the corner $(0,0)$ decreases. 
The player that gets the Queen into the corner wins.
The reach of the Queen is most beneficial for this game,
which would take much longer with a Counselor.
 
To win the game, Anne could certainly place the Queen on $(0,0)$,
but that is not very entertaining.
What if she moves further out in her realms? 
The set of all winning squares, formally known as $P$-positions (the complement of which are the $N$-positions), form a
fascinating pattern that continues to be a source of inspiration
for the study of combinatorial games~\cite{DFGHKL}.
Throughout the paper we will reserve $(a_n,b_n)$ for the $P$-positions
and by symmetry we may restrict ourselves to $a_n\leq b_n$.
The positions are lined up by their first coordinate so
that $a_n$ is increasing. It turns out that $b_n$ is increasing as well 
for all games in this paper
(this does not apply to all versions of Wythoff's game, see for instance~\cite{L})
and that the sets $\{a_n\colon n\in\mathbb N\}$ and $\{b_n\colon n\in\mathbb N\}$
are complementary if we ignore the first $P$-position $(0,0)$.
Wythoff~\cite{W} showed for his original game
that $a_n=\lfloor n\phi\rfloor$ and
$b_n=a_n+n$ for the golden mean~$\phi$, which
satisfies $\phi^2=\phi+1$.
Holladay~\cite{H} considered what happens if the Queen can stroll slightly off diagonal.
More specifically, Holladay's Queen can move from $(x,x+c)$ to $(x',x'+d)$ if $x'<x$ and
$|c-d|\leq k-1$ for $k\in\mathbb N$.
Let's call such a Queen a $k$-Queen.\footnote{In fact, Holladay also considered
other kinds of Queen that lead to the same $P$-positions.}  For $k=1$ she is the original Queen.   
Holladay proved that the $P$-positions 
for Corner the $k$-Queen 
satisfy $a_n=\lfloor n\psi_k	\rfloor$
and $b_n=a_n+kn$ for quadratic numbers $\psi_k$ that have 
continued fraction expansion $[1;k,k,k,\ldots]$.

In this paper, we consider Queens that have a greater mobility
than the standard Queen. A good name for games in which they
occur seems to be \emph{Corner the Empress}, since Empresses rank above Queens.
Eric Duch\^ene and Michel Rigo introduced the idea of a morphic game~\cite{DR} (see Section 4).
Together with Julien Cassaigne they showed that almost all Beatty sequences
(homogeneous or not) arise from the set of $P$-positions in certain take-away games~\cite{CDR}.
One motivation for our present paper is a quest for more morphic games.
All Empresses that we consider in this paper lead to morphic games.
Some are old and some are new. All are getting cornered,
which seems to be the way to go for royals in our modern age.

Our paper is a sequel to earlier work~\cite{FR} in which we introduced
the game of Splithoff. This game is closely related to the Queens on a Spiral
problem that was solved in~\cite{DSS}.
We left two conjectures in our earlier paper, one of which 
we solve in Section~3.
The other conjecture, which we consider in Section~5,
was recently solved by Jeffrey Shallit~\cite{S2} using the automatic theorem prover \texttt{Walnut}. This is an elegant and powerful tool that can be used with great effect in the analysis of morphic sequences. 
Indeed, we shall show in Section~6 that \texttt{Walnut} is also able to prove the other
conjecture from our earlier paper.
We encourage the reader to check out the new and highly entertaining textbook~\cite{S1} for many more examples of the power of \texttt{Walnut}.

\section{The Queen Bee}

Suppose the Queen has the ability to bounce, or perhaps more aptly: 
if Her Majesty starts out diagonally, it may please her
to continue her way perpendicularly
once she reaches the boundary of her domain,
as long as the Manhattan distance decreases.
Because of this bounciness, we shall call her the
Queen Bee.
\begin{figure}[htbp]
	\centering
	\includegraphics[width=0.35\textwidth]{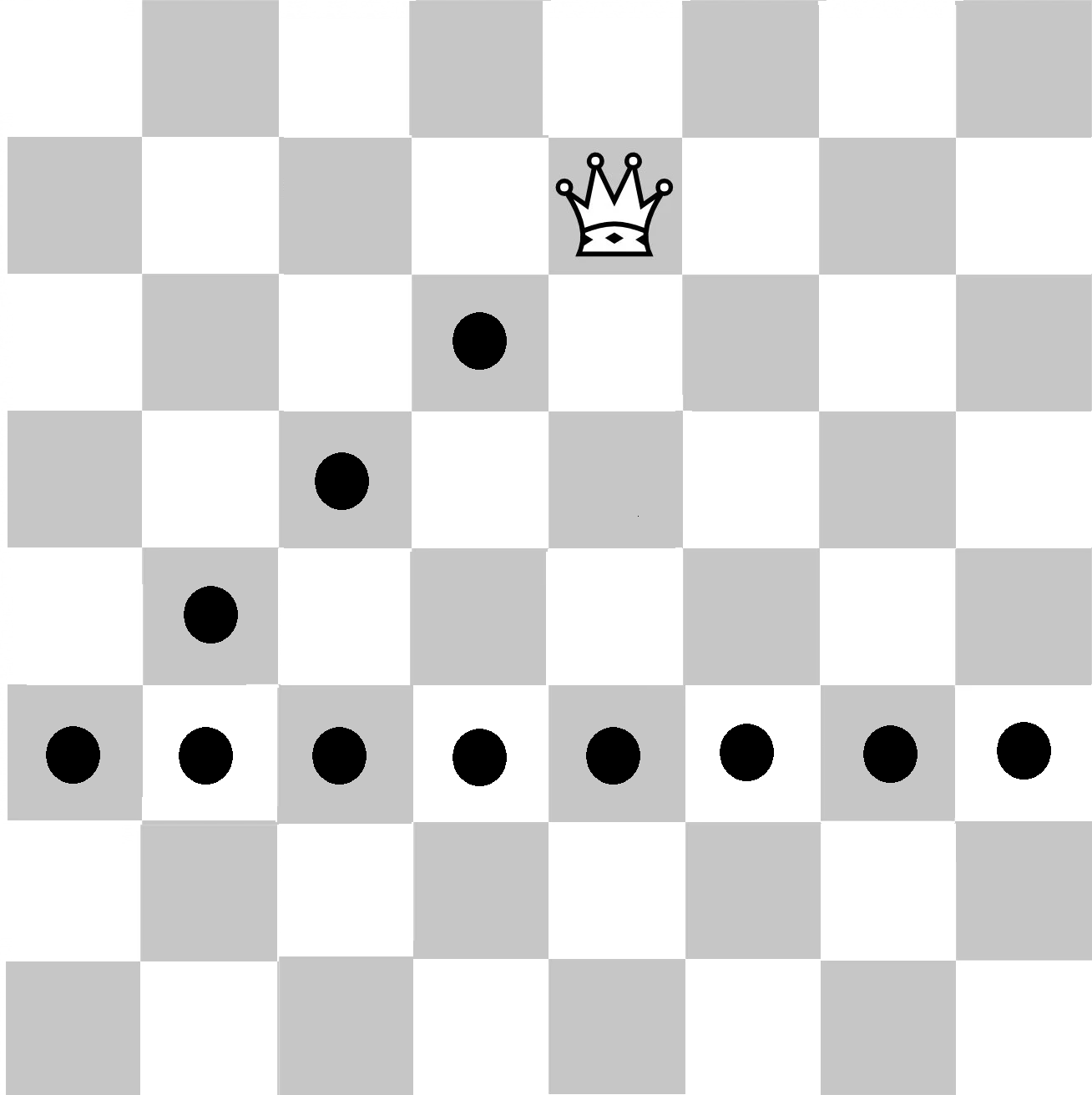}
	\caption{Possible moves of a Queen Bee that can bounce perpendicularly (black dots),
	if Her Majesty starts out
	on a diagonal.}
	\label{fig1}
\end{figure}
The first few $P$-positions of Corner the Queen Bee 
are given in Table~\ref{tbl1}.
Its defining property is that all natural
numbers occur once, and only once, as an $a_n$ or a $b_n$ and that
$b_n=2a_n$. 
It turns out that this Table also arises from a subtraction game
considered by Fraenkel~\cite{Fsub}, who calls the $a_n$ vile
and the $b_n$ dopey.

\begin{table}[htbp]
\caption{\small{The first entries
in the table of $P$-positions in Corner the Queen Bee. All coordinates
satisfy $b_n=2a_n$.
The $a_n$ are the numbers whose binary representations end with an even number of zeros. 
The $b_n$ are those that end with an odd number of zeros.
Notice that
all positions are on diagonals $(x,x+a)$ for which
$a$ is in the top row.
These are the sequences A003159 and A036554
in the OEIS.
}\label{tbl1}}
\begin{tabular}{cccccccccccccccc}
\hline
1&3&4& 5& 7& 9&11&12&13&15&16&17&19&20&21  \\
2&6&8&10&14&18&22&24&26&30&32&34&38&40&42 \\
\hline
\end{tabular}
\end{table}

\begin{theorem}\label{thm1}
The $P$-positions of Corner the Queen Bee satisfy
$b_n=2a_n$ and each natural number occurs once, and only once, as
a coordinate of a $P$-position.
\end{theorem}

\begin{proof}
It is not hard to see that every column and every row contain at most
one $P$-position. 
To see that every column (and hence every row) contains a
$P$-position, observe that all
positions within reach of $(a,y)$ have $x$-coordinate bounded
by $2a$. 
This is immediately clear for non-bouncing moves.
A bouncing move leads to a position $(a',y-a)$ such that
$a'<2a$, because the Manhattan distance decreases.
Therefore, the number of $P$ positions within reach of $(a,y)$
is at most $2a$ and one of these positions must be reachable
by infinitely many $(a,y)$. This is only possible by
vertical moves. We conclude that the column contains a $P$-position.
It follows that every natural number occurs once,
and only once, as an $a_n$ or a $b_n$. 

The minimal excludant of $\{a_i,b_i\colon i\leq n\}$ occurs as an $a$ or
a $b$. It has to be $a_{n+1}$ since we line up $P$-positions
by increasing $x$-coordinate. Now consider
an arbitrary $(a_{n+1},a_{n+1}+d)$ for $0< d<a_{n+1}$.
A diagonal move bounces against the side at $(0,d)$ with
$d\in\{a_i,b_i\colon i\leq n\}$. If $d$ is a vile number, then
by induction
$(d,2d)$ is a $P$-position. It is within reach of $(a_{n+1},a_{n+1}+d)$
since $d<a_{n+1}$.
If $d$ is dopey, then the $P$-position $(d/2,d)$ is within reach.
We conclude that $(a_{n+1},a_{n+1}+d)$ is an $N$-position.
It is easy to check that none of the $(a_i,b_i)$ are within
reach of $(a_{n+1},2a_{n+1})$. Hence, this is a $P$-position.    
\end{proof}

\section{The Queen Dee}

In our earlier work~\cite{FR} we
considered a Queen that can reflect once against a side and continue
diagonally, changing her course from diagonal (constant $y-x$) to anti-diagonal 
(constant $y+x$) if she
bounces against a side of the board. 
We call her a Queen Dee.
In this paper, we consider a $2$-Queen with the same reflective property
and we call her a $2$-Queen Dee.
A $2$-Queen can take one step to the left or one step down before moving
diagonally. A $2$-Queen Dee can reflect off of a side, see~Fig.~\ref{fig2}.
\begin{figure}[htbp]
	\centering
		\includegraphics[width=0.80\textwidth]{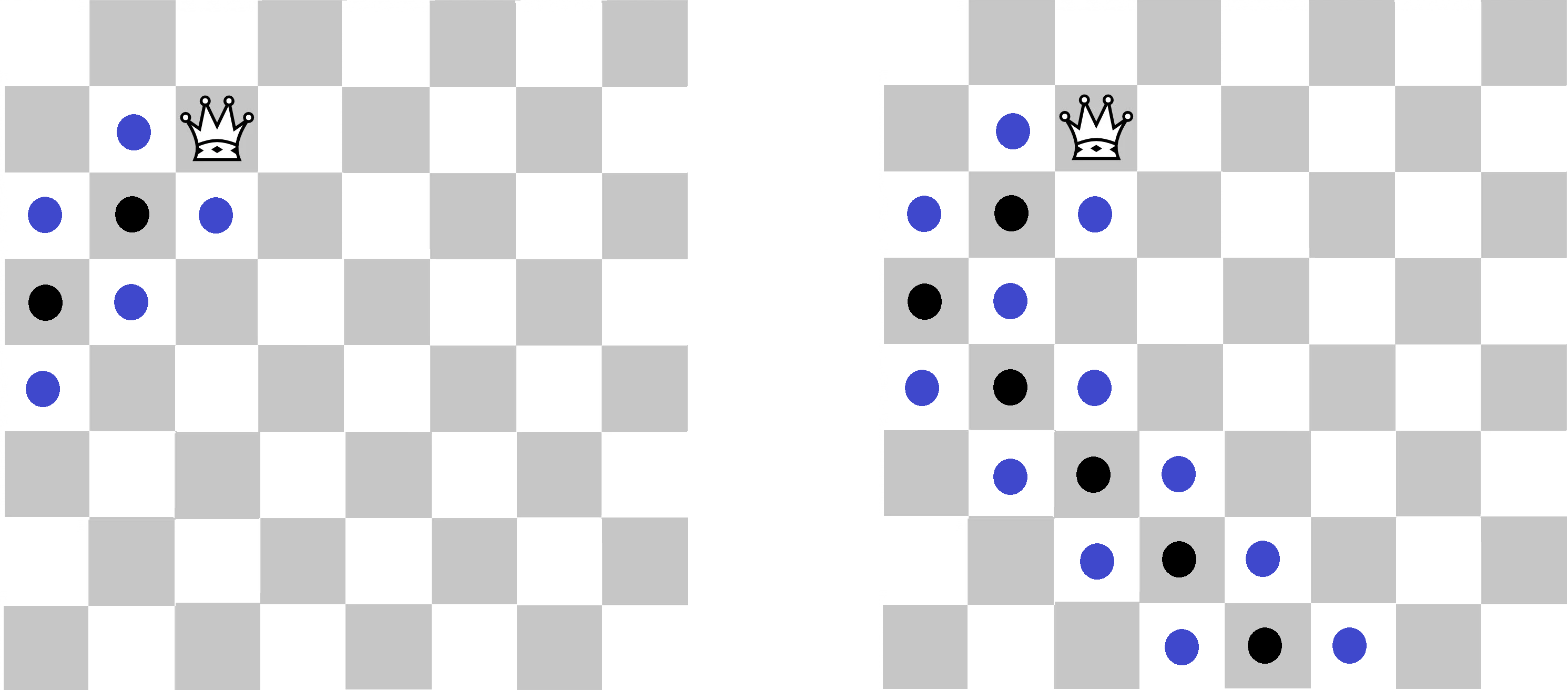}
	\caption{On the left, possible diagonal moves of a Queen (black dots) and a $2$-Queen
 (additional blue dots).
         These moves are extended by reflection for the Queen Dee and the $2$-Queen Dee,
         as illustrated on the right. }
	\label{fig2}
\end{figure}

We showed in~\cite{FR} and~\cite{O}
that the $P$-positions in Corner the Queen Dee can be retrieved
from the morphism $\tau(a)= ab,\ \tau(b)=ac,\ \tau(c)=a$, as follows.
If we iterate this morphism, 
then the limit is the Tribonacci word
\[
\mathbf{t} = \lim_{n\to\infty}\tau^n(a) = abacabaabacababacabaabacab\cdots,
\]                                        
which is an intriguing object
that is studied for its own sake~\cite{RSZ,TW}. If we delete
all $c$'s from $\mathbf t$ then we get the word
$\mathbf t_c=abaabaabaababaabaabaab\cdots$
which turns out to code the $P$-positions $(a_n,b_n)$.
More specifically, the location of the $n$-th $a$ in $\mathbf t_c$ is equal to $a_n$
and the location of the $n$-th $b$ is equal to $b_n$.
In this section, we complement this
result by proving that the $n$-th $P$-position in Corner the $2$-Queen Dee can be retrieved
in the same manner from $\mathbf t_b=aacaaaacaaacaaaaca\cdots$, the Tribonacci word from which all $b$'s have
been deleted.

Recently, Jeffrey Shallit~\cite{S2} gave a mechanical proof using \texttt{Walnut}
that the $P$-positions of the Queen Dee are coded by the Tribonacci word.
It is straightforward to adapt this proof for $P$-positions of the $2$-Queen Dee
and we shall do that in Section~6 below. In this section we give an old-fashioned proof.
\medbreak
The $P$-positions of the $2$-Queen satisfy
$b_n=a_n+2n$, i.e., they are on consecutive even diagonals. 
The $P$-positions of the $2$-Queen Dee
occasionally skip an even diagonal,
as shown in Table~\ref{tbl2}. 
\begin{table}[htbp]
\caption{\small{The first $P$-positions
of the $2$-Queen Dee are all on even diagonals $(x,x+2d)$.
Notice that 
that $2d$ is unequal to $a_n+b_n$ for all $n$.
These are sequences A140102 and A140103 in
the OEIS.
}\label{tbl2}}
\begin{tabular}{cccccccccccccccc}
\hline
1&2&4& 5& 6& 7 & 9&10&11&13&14&15&16&18&19  \\
3&8&12&17&20&25&29&34&39&43&48&51&56&60&65 \\
\hline
\end{tabular}
\end{table}
In particular, the positions skip the diagonals 
$(x,x+2d)$ if $2d$ is equal to $a_1+b_1,\ a_2+b_2,\ a_3+b_3,$ etc.
It turns out that this is a defining property.
\medbreak
The standard algorithm to find $P$-positions in impartial games starts by
identifying the set $P_0$ of all (or at least one) 
end-positions, which have no move. 
In our case,
this is $(0,0)$, the $P$-position $(a_0,b_0)$
that we never list in our tables of $(a_n,b_n)$.
Now remove all positions that have a move to $P_0$, and consider the 
resulting game.
Identify its set of end-positions and add them to get 
$P_1$. Remove all positions that have
a move to $P_1$, etc.
Continue until all positions are either in $P_n$ or have been removed.
The union of the $P_n$ is the set of $P$-positions.

To apply this algorithm to the $2$-Queen Dee we introduce some notation:
\begin{eqnarray*}
&A_n=\{a_i\colon i\leq n\},\ B_n=\{b_i\colon i\leq n\}, P_n=\{(a_i,b_i)\colon i\leq n\}   \\ 
&D_n=\{b_i-a_i\colon i\leq n\}, \ S_n=\{a_i+b_i\colon i\leq n\}
\end{eqnarray*}
These are all initial segments of sequences, of coordinates, positions, differences, and sums.
To find the next $P$-position, we need to remove all positions with a move to $P_n$ from
the game. 

Table~\ref{tbl2} indicates that $A_n\cup B_n$ contains all integers up to $\max(A_n)$
and that $D_n\cup S_n$ contains all even integers up to $\max(D_n)$, while both
$D_n$ and $S_n$ contain even integers only. This is our inductive hypothesis, together
with the assumption that $P_n$ contains the first $n+1$ $P$-positions (if we include
the origin) with increasing Manhattan distance to the origin.

\begin{lemma}\label{lem1}
Under the inductive hypothesis 
there is a move from $(x,y)$ with $x\leq y$ to $P_n$ if and only if one of the following holds:
$\{x,y\}\cap \left(A_n\cup B_n\right)\not=\emptyset$ or $\{y-x-1,y-x,y-x+1\}\cap \left(D_n\cup S_n\right)\not=\emptyset$.
\end{lemma}
\begin{proof}
Suppose that there exists a move from $(x,y)$ to $(a_i,b_i)\in P_n$
(or its reflection $(b_i,a_i)$).
If the move is horizontal or vertical, then the $P$-position
has to share a coordinate with $(x,y)$. In other words,
$\{x,y\}\cap (A_n\cup B_n)$ is non-empty. 
If the move is diagonal without a bounce, then $y-x=b_i-a_i$
If it is diagonal with a bounce, then $y-x=a_i+b_i$.
Similarly, if the move is off diagonal starting from $(x,y-1)$
or $(x-1,y)$
then $y-x\pm 1=b_i-a_i$ or $y-x\pm 1=b_i+a_i$.
We conclude that a move from $(x,y)$ to $P_n$ is only possible
if one of the two intersections 
$\{x,y\}\cap (A_n\cup B_n)$ or $\{y-x-1,y-x,y-x+1\}\cap \left(D_n\cup S_n\right)$
is non-empty.

We need to prove the converse.
First observe that if $(x,y)$ is reachable from $P_n$, then it has a move to $P_k$
for $k<n$, because any position that is reachable from a $P$-position is an $N$-position.
So all we have to show is that a move between $(x,y)$ and~$P_n$ is possible
if one of the two intersections is non-empty.
Suppose that $\{x,y\}\cap \left(A_n\cup B_n\right)\not=\emptyset$, i.e., there exists a $P$-position
$(a_i,b_i)$ (or its reflection $(b_i,a_i)$)
that is in the same row or column as $(x,y)$. Hence there exists a move
between $(x,y)$ and $(a_i,b_i)$ or its reflection.

Suppose that $\{x,y\}\cap \left(A_n\cup B_n\right)=\emptyset$
and that $\{y-x-1,y-x,y-x+1\}\cap \left(D_n\cup S_n\right)\not=\emptyset$.
If $y-x$ is even,
then $(y-x)\in \left(D_n\cup S_n\right)\not=\emptyset$. 
We have that
$a_i\pm b_i=y-x$ for a proper choice of the sign. If the sign is
negative, then the positions are on the same diagonal and there is
a move between them.
If the sign is positive, then there is a diagonal move with a bounce that takes
$(y,x)$ to $(a_i,b_i)$. 
If $y-x$ is odd then either $(y-x-1)$ or $(y-x+1)$ is in $\left(D_n\cup S_n\right)$.
By the same argument, there either exists a diagonal move (possibly with a bounce)
from $(x,y-1)$ or from $(x-1,y)$ to some $(a_i,b_i)$.
\end{proof}

Recall that the minimal excludant of a proper subset $S$ of the natural numbers is denoted 
$\mathrm{mex}(S)$.
For a proper subset $S$ of the even natural numbers we define $\mathrm{mex}_2(S)$ as the minimal
excludant in $2\mathbb N$. For example, if $S=\{2,4,6,10,14\}$ then $\mathrm{mex}_2(S)=8$.

\begin{lemma}\label{lem2}
The coordinates of the $n+1$-th $P$-position satisfy $a_{n+1}=\mathrm{mex}(A_n\cup B_n)$
and $b_{n+1}-a_{n+1}=\mathrm{mex}_2(D_n\cup S_n)$.
\end{lemma}
\begin{proof}
To find the next $P$-position, remove all positions with a move to $P_n$. 
The remaining positions $(x,y)$ have empty intersections 
$\{x,y\}\cap (A_n\cup B_n)=\emptyset$ and $\{y-x-1,y-x,y-x+1\}\cap \left(D_n\cup S_n\right)=\emptyset$. 
We need to locate the remaining position with minimal Manhattan distance to the origin.
The minimal remaining $x$-coordinate is $\mathrm{mex}(A_n\cup B_n)$.
The minimal remaining difference $y-x$ is $\mathrm{mex}_2(D_n\cup S_n)$.
Note that the resulting $A_{n+1},B_{n+1},D_{n+1},S_{n+1}$ satisfy the inductive hypothesis.
\end{proof}

There is a remarkable connection between the $P$-positions for
Queen Dees and for $2$-Queen Dees.
\begin{table}[htbp]
\caption{\small{The first $P$-positions
in Corner the Queen Dee.
These are sequences A140100 and A140101 in the OEIS.
Observe that A140101-A140100=A140103, which is the first sequence
in Table~\ref{tbl2} above, and that A140101+A140100=A140104,
which is the second sequence.
It immediately follows that A140104-A140103=2A140100.
}\label{tbl3}}
\begin{tabular}{cccccccccccccccc}
\hline
1&3&4&6&7&9&10&12&14&15&17&18&20&21&23 \\
2&5&8&11&13&16&19&22&25&28&31&33&36&39&42 \\
\hline
\end{tabular}
\end{table}
The differences $b_n-a_n$ for the Queen Dee
produce the sequence $1,2,4,5,6,7,\ldots$ of $x$-coordinates
of $P$-positions of the $2$-Queen Dee.
Similarly, the sums $b_n+a_n$ produce
$3,8,12,17,20,25,\ldots$ of $y$-coordinates for the $2$-Queen Dee.
We can retrieve the $P$-positions of the $2$-Queen Dee from those
of the Queen Dee. Conversely, rewriting these equations, we
can retrieve the $P$-positions of the Queen Dee from those of the $2$-Queen Dee.
The differences $b_n-a_n$ of the $2$-Queen Dee produce
the sequence $2, 6, 8, 12, 14, 18, \ldots$ of twice the $x$-coordinates
of $P$-positions of the Queen Dee.
Similarly, the sums produce twice the $y$-coordinates.

We write $d_n=b_n-a_n$ for difference and $s_n=b_n+a_n$ for sum.
We reserve the Greek alphabet for the Queen Dee and Latin for the $2$-Queen Dee.

\begin{lemma}\label{lem3}
The following equalities hold $a_n=\delta_n, b_n=\sigma_n, d_n=2\alpha_n, s_n=2\beta_n$. 
\end{lemma}

\begin{proof}
We use the Greek notation $\Delta_n=\{\delta_i\colon i\leq n\}$
and $\Sigma_n=\{\sigma_i\colon i\leq n\}$.
We proved in \cite{FR} that $\delta_{n+1}=\mathrm{mex}(\Delta_n\cup\Sigma_n)$, which
by induction and by Lemma~\ref{lem2} is equal to $a_{n+1}$,
producing the first equality.

Unfortunately, the Greek capitals $\alpha$ and $\beta$ are indistinguishable
from the Latin capitals. Fortunately, by induction,
we can identify the first $n$ $\alpha$'s
with $D_n/2$ and the first $n$ $\beta$'s with $S_n/2$.  
We proved in \cite{FR} that $\alpha_{n+1}=\mathrm{mex}(D_n/2\cup S_n/2)$
and by Lemma~\ref{lem2} this implies that $\alpha_{n+1}=d_{n+1}/2$,
producing the third equality.

Moving on, $b_{n+1}=a_{n+1}+d_{n+1}=\delta_{n+1}+2\alpha_{n+1}=\alpha_{n+1}+\beta_{n+1}$
produces the second equality. Finally, $s_{n+1}=a_{n+1}+b_{n+1}=\delta_{n+1}+\sigma_{n+1}=2\beta_{n+1}$
produces the fourth equality.
\end{proof}

We conjectured in~\cite{FR} that the $P$-positions of the $2$-Queen Dee
can be derived from~$\mathbf t$. Our next result settles this conjecture.

\begin{theorem}\label{thm2}
The $P$-positions $(a_n,b_n)$ in Corner the $2$-Queen Dee are coded by $\mathbf t_b$,
in the sense that $a_n$ corresponds to the location of the $n$-th $a$ and $b_n$
corresponds to the location of the $n$-th $c$.
\end{theorem}

\begin{proof} Again this follows from results in \cite{FR}. Corollary~1 in that
paper states that $\delta_{n+1}-\delta_n=1$ unless the $n$-th letter
in $\mathbf t$ is equal to $b$, in which case $\delta_{n+1}-\delta_n=2$.
Therefore, we can retrieve $\delta_n$ as the position of the $n$-th $a$ if
we code $\mathbf t$ by $a\mapsto a,\ b\mapsto ac,\ c\mapsto a$.
We denote this coding by $\kappa$. In particular, $\delta_n$ corresponds to
the $n$-th $a$ in $\kappa(\mathbf t)$.
Note that $\kappa$ is equal to $\lambda\circ\tau$ where $\lambda$ is the coding
$a\mapsto a,\ b\mapsto \epsilon,\ c\mapsto c$.
Now the result follows from
\[\mathbf t_b=\lambda(\mathbf t)=\lambda(\lim_{n\to\infty}\tau^{n+1}(a))=\kappa(\lim_{n\to\infty}\tau^n(a))=\kappa(\mathbf t)\]
and the fact that by Lemma~\ref{lem3} the sequences $\delta_n$ and $a_n$ are identical.
\end{proof}

One may wonder what happens if we delete all $a$'s from $\mathbf t$. It turns out that
$\mathbf t_a$ codes the same table as $\mathbf t_c$. 

\section{Morphic games}

The $P$-positions in Corner the Lady can be derived from the $a$'s and
the $b$'s in the Fibonacci word, which is the fixed point of the morphism
$a\mapsto ab, \ b\mapsto a$. The fixed point of the morphism 
$a\mapsto a^kb,\ b\mapsto a$ produces the $P$-positions for Holladay's $k$-Queen
(here $a^k$ denotes the block of $k$ $a$'s). The period
doubling morphism $a\mapsto ab,\ b\mapsto aa$ produces the $P$-positions
for the Queen Bee. 
This connection between morphisms and combinatorial games was first 
studied by Duch\^ene and Rigo~\cite{DR}, see also~\cite{R}. A game is \emph{morphic} if its
$P$-positions can be retrieved from a 
recoding of a fixed point of a morphism.
The $k$-Queen, the Queen Bee, and the $1$ and $2$-Queen Dee all are morphic Queens.
Are there more?

Not all Bouncing Queens are morphic.	
In an earlier version of \cite{FR} two authors of this paper conjectured that $3$
and $4$-Queen Dees are morphic.
However, numerical results of the other author~\cite{O} indicate otherwise.

Not all modifications of the substitutions above produce morphic Queens. 
The morphisms $a\mapsto ab,\ b\mapsto a$ and $a\mapsto ab,\ b\mapsto aa$,
which come with the Queen and the Queen Bee, belong to the family
$a\mapsto ab,\ b\mapsto a^k$ that is studied in~\cite{BGM} (in an equivalent form).
Next in line is $a\mapsto ab,\ b\mapsto aaa$. It produces the positions
$(1,2),\ (3,7),\ (4,9),\ (5,11),\ (6,13)$. 
Now consider the $N$-position $(3,6)$.
It must be possible to move to a $P$-position from there, which has
to be either $(0,0)$ or $(1,2)$. In the first case, we need a move over
$(3,6)$, which is the difference between the $P$-positions $(3,7)$ and $(6,13)$.
In the second case, we need a move over $(2,4)$, which is the difference
between the $P$-positions $(3,7)$ and $(5,11)$.
Hence, there is no sensible way to define a Queen in this case for which legal moves are position-independent,
and so it appears that $a\mapsto ab,\ b\mapsto aaa$ does not produce a morphic Queen.

Some modified substitutions do produce morphic Queens.
The morphism $a\mapsto aab,\ b\mapsto aa$ arises from the $2$-Queen's morphism, 
just like the Queen Bee's morphism arises from the Queen's.
It produces the $P$-positions in~Table~\ref{tbl4}.
\begin{table}[htbp]
\caption{\small{The $P$-positions of the $(2,1)$-Queen
are
produced by the morphism $a\mapsto aab,\ b\mapsto aa$.
The positions $(a_n,b_n)$ satisfy $b_n=2a_n+n$.
These are sequences A026367 and A026368 in the OEIS.
}\label{tbl4}}
\begin{tabular}{cccccccccccccccc}
\hline
1&2&4& 5& 7& 8 &9 &10&12&13&15&16&17&18&20 \\
3&6&11&14&19&22&25&28&33&36&41&44&47&50&55 \\
\hline
\end{tabular}
\end{table}
These are the $P$-positions of $(2,1)$-Wythoff, a modification of
Corner the Lady that was introduced by Fraenkel~\cite{Fheap}.
In this game, the Queen widens the scope of her diagonal. 
An ordinary Queen makes diagonal moves
$(x,y)$ with $x$ equal to $y$. 
In $(2,1)$-Wythoff, we have a $(2,1)$-Queen who makes
widened diagonal moves
$x\leq y\leq 2x$,  
or symmetrically, $y\leq x\leq 2y$.
Consecutive positions in Table~\ref{tbl4} either differ by
$\left[\begin{smallmatrix} 1\\3\end{smallmatrix}\right]$ or   
$\left[\begin{smallmatrix} 2\\5\end{smallmatrix}\right]$, and contain
each natural number once and only once.
It is not hard to verify that for $k>1$ such tables with differences
$\left[\begin{smallmatrix}1\\{k+1}\end{smallmatrix}\right]$ or   
$\left[\begin{smallmatrix}2\\{j+k+1}\end{smallmatrix}\right]$
correspond to the morphisms $a\mapsto a^kb,\ b\mapsto a^j$. 
These tables contain the $P$-positions of Fraenkel's
$(j,1+k-j)$-Wythoff games (if $k\geq j$), which can be interpreted as versions of Corner
the Lady for Queens with widened scopes.
The first parameter $j$ represents the scope, the second parameter
represents how far the Queen can stroll off the diagonal.
In particular, a $(1,k)$-Queen is a $k$-Queen.

An ordinary Queen produces complementary sequences $b_n=a_n+n$. Holladay's
$k$-Queen produces $b_n=a_n+kn$. What can be said for complementary sequences
that satisfy $b_n=a_n+f(n)$ for some function~$f$? This problem has been
studied by Kimberling~\cite{K} and some of these sequences have been entered
in the OEIS (see A184117). The morphism $a\mapsto a^{k-j}ba^j,\ b\mapsto a$
produces $b_n=a_n+kn-j$ for $k>j\geq 0$. These are $P$-positions. Consider
a $k$-Queen that is allowed to stroll, but with a restriction. If her initial
position $(x,y)$ satisfies $|x-y|>j$ then she cannot stroll off-diagonal 
to a position $(u,v)$ with $|u-v|\leq j$. The positions $|u-v|\leq j$
form a special part of the board that can only be entered by a horizontal
or vertical move. The $P$-positions in this limited game are produced
by~$a\mapsto a^{k-j}ba^j,\ b\mapsto a$. 

The interesting quality of the Queen Dee is that her $P$-positions are coded
by the fixed point of 
the three letter morphism $a\mapsto ab,\ b\to ac,\ c\mapsto a$. 
This morphism belongs
to a well-studied family of morphisms, in which the $4$-bonacci
morphism 
$a\mapsto ab,\ b\mapsto ac,\ c\mapsto ad,\ d\mapsto a$
is next. In their original paper on morphic games, Duch\^ene and Rigo
asked if it is possible to define a combinatorial game with $P$-positions
corresponding to the $4$-bonacci morphism. We can specify this. 
Does there exist a morphic Queen with $P$-positions
that are coded by the fixed point of this morphism? 	

\section{The coordinate sums of the \textit{P}-positions}

One of the remarkable properties of Corner the Lady is that consecutive Fibonacci numbers
$(F_i, F_{i+1})$ occur as $P$-positions in this game (with $i$ odd and starting the count
from $F_1$=1 and $F_2=2$). Martin Gardner singled out this property in one of his Scientific
American columns~\cite{G}. It is a consequence of the equation $a_{b_n}=a_n+b_n$ for
the $P$-positions of Corner the Lady.
This equation does not hold for Corner the Empress, but it holds approximately, as can be 
seen in Table~\ref{tbl3}.
Its first column adds up to $1+2=3$, which is indeed equal to the first entry in the second column. 
However, the second column sum $3+5=8$ is not equal to the first entry of the fifth column.
It is one off, $7$. Further inspection indicates that $a_n+b_n$ appears to be either equal to the $a$ entry in the $b_n$-th column, or is one off.
In other words
\begin{equation}\label{conj1}
\forall\,n\ \left( a_{b_n}=a_n+b_n\ \vee\
a_{b_n}=a_n+b_n-1\right). 
\end{equation}
In \cite{FR} we conjectured that this equation holds for Corner the Queen Dee, and Jeffrey Shallit used \texttt{Walnut} to prove it in~\cite{S2} by the following one-liner:

\noindent
\begin{center}
\texttt{eval\ fokkink\ "?msd\_trib\ An,a,b,ab\ (n>=1\ \&\ \$xaut(n,a)\ \&\ \$yaut(n,b)\ \&\ }
\texttt{\$xaut(b,ab))\ =>\ (ab=a+b|ab=a+b-1)":}
\end{center}

\noindent
\texttt{Walnut} code is very transparent. Even if you are not familiar with the language, you can understand this statement. The final expression
$\mathtt{(ab=a+b\ |\ ab=a+b-1)}$ contains the conjecture. The sequences $\mathtt a$ and $\mathtt b$ are defined by means of
$\mathtt{\$xaut(n,a)}$ and $\mathtt{\$yaut(n,b)}$. A sequence is a function from the integers to itself and the two functions for $a_n$ and $b_n$ are $\mathtt{xaut}$ and $\mathtt{yaut}$. Actually, \texttt{Walnut} does not write $\mathtt {a=xaut(n)}$, or something like that, but treats $\mathtt n$ and $\mathtt a$ as simultaneous inputs. This is why $\mathtt n$ and $\mathtt a$ are called synchronized sequences. Since \texttt{Walnut} can only handle bitstrings, it needs a numeration system to deal with natural numbers. The default is the binary system, but it turns out that $\mathtt n$ and $\mathtt a$ are not synchronized in binary. The proper numeration system is the Tribonacci system, which is imported by the command $\mathtt{msd\_trib}$. The instruction \texttt{eval} commands \texttt{Walnut} to evaluate this statement, named after one of the authors, for all $n, a, b, ab$ with $n\geq 1$. The statement is evaluated as \texttt{TRUE}.

Table~\ref{tbl2} appears to also satisfy equation~\ref{conj1}. To apply \texttt{Walnut} we need to implement the functions for the sequences $a_n$ and $b_n$ in this table. Fortunately, they can easily be derived from the synchronized sequences in \cite[p 286] {S1}, which have been implemented in \texttt{Walnut} to reprove a theorem from~\cite{DR}. We have that $a_n=A(n)-E(A(n))$ and $b_n=C(n)-E(C(n))$. The sequences $A,C,E$ are implemented as \texttt{triba, tribc, tribe}, which we can use to define
\medbreak
\noindent
\texttt{def\ uaut\ "?msd{\textunderscore}trib\ (s=0\&t=0)\ |\ Ex,xy\ \$triba(s,x)\&\$tribe(x,xy)\&t=x-xy":}
\noindent
\texttt{def vaut\ "?msd{\textunderscore}trib\ (s=0\&t=0)\ |\ Ex,xy\ \$tribc(s,x)\&\$tribe(x,xy)\&t=x-xy":}
\medbreak
\noindent
Now we can copy-paste the proof in \cite{S2} to show that the columns in Table~\ref{tbl2} do indeed satisfy equation~\ref{conj1}.
\begin{theorem}
The $P$-positions $(a_n,b_n)$ of the 2-Queen Dee satisfy Equation~\ref{conj1}.
\end{theorem}
\begin{proof}
\noindent
\begin{center}
\texttt{eval\ table2\ "?msd\_trib\ An,a,b,ab\ (n>=1\ \&\ \$uaut(n,a)\ \&\ \$vaut(n,b)\ \&\ }
\texttt{\$uaut(b,ab))\ =>\ (ab=a+b|ab=a+b-1)":}\end{center}
\noindent
\texttt{Walnut} returns \texttt{TRUE}. It befits a modern Monarch to have her footsteps traced by an electronic slave.
\end{proof}

Numerical experiments indicate that $a_{b_n}-a_n-b_n$ is unbounded for the Queen Bee. We cannot use \texttt{Walnut} to prove or disprove this statement. Narad Rampersad and Manon Stipulanti~\cite{RS} proved that the sequences $a_n$ and $b_n$ are not regular in this case, which implies that they cannot be synchronized. 
\begin{figure}[htbp]
	\centering
		\includegraphics[width=0.5\textwidth]{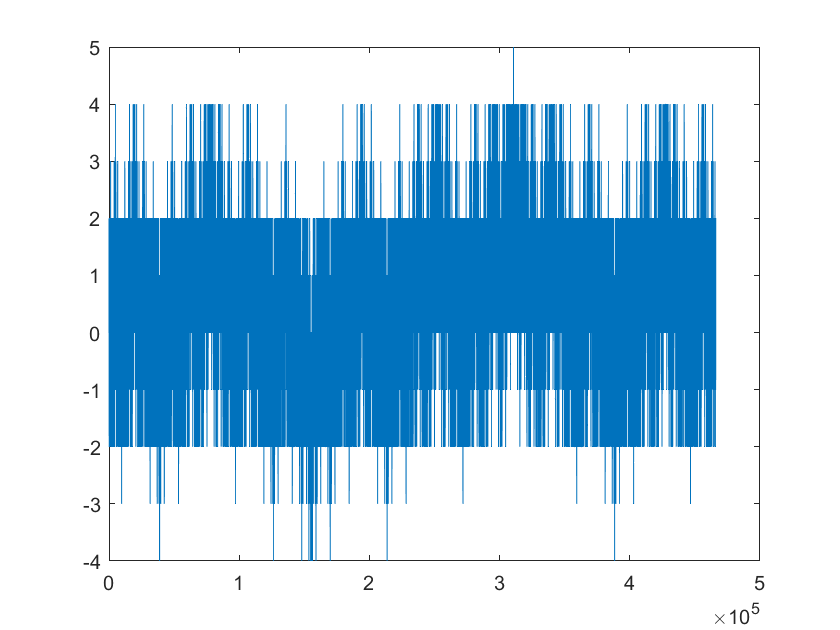}
	\caption{The difference $a_{b_n}-a_n-b_n$ plotted against $n$ for the $P$-position in Corner the Queen Bee. The first difference that is equal to $5$ occurs at $n=310691$.}
	\label{fig3}
\end{figure}
\begin{proposition}\label{prop1}
Let $[x]$ denote the nearest integer to $x$, rounded up if the fractional part is $\frac 12$ and let \[f(x)=\sum_{j=0}^\infty \left[\frac x{2^{2j+1}}\right].\] Then $n=f(a_n)$ for the $n$-th $P$-position in Corner the Queen Bee.
\end{proposition}

\begin{proof}
The sequence $a_n$ consists of all numbers that are $4^{j}\mod 2.4^{j}$ for $j=0,1,\ldots$,
since these are the numbers that have a suffix of one $1$ followed by $2j$ zeros. 
The residue class $2^{2j}\mod 2^{2j+1}$ contains $\left\lfloor \frac {a_n+2^{2j}}{2^{2j+1}}\right\rfloor$ elements up to $a_n$. Therefore, $a_n$ is the $n$-th element with 
\[
n=\sum_{j=0}^\infty \left\lfloor\frac {a_n+2^{2j}}{2^{2j+1}}\right\rfloor.
\]
Now notice that $[x]=\left\lfloor x+\frac 12\right\rfloor$.
\end{proof}

We are interested in $a_{b_n}-a_n-b_n$, which for the Queen Bee is equal to $a_{2a_n}-3a_n$.
Now $a_{2a_n}$ has index $2a_n$ while $3a_n$ has index $f(3a_n)$ by Proposition~\ref{prop1} 
(both sequences $a_n$ and $b_n$ are invariant under multiplication by $3$). 
To prove that $a_{2a_n}-3a_n$ is unbounded, we need to show that $2a_n-f(3a_n)$
is unbounded. It is
equal to
\[
2a_n-\sum_{j=0}^\infty \left[\frac {3a_n}{2^{2j+1}}\right].
\]
Note that this is equal to zero if we do not round the fractions $\frac {3a_n}{2^{2j+1}}$.
In other words, the unboundedness of $a_{b_n}-a_n-b_n$ is equivalent to the unboundedness
of
\begin{equation}\label{eq2}
\sum_{j=0}^\infty \left\{\frac {3a_n}{2^{2j+1}}\right\},
\end{equation}
where $\left\{x\right\}$ denotes the distance to the nearest integer $x-[x]$ (beware, it
is not the fractional part $x-\lfloor x\rfloor$!).
\begin{theorem}
The difference $a_{b_n}-a_n-b_n$ is unbounded for the $P$-positions of the  Queen Bee.
\end{theorem}
\begin{proof}
Define $h(x)=\sum_{j=0}^\infty \left\{\frac {x}{4^{j}}\right\}$. Observe that
$h(4^n)=\frac 13$ and that for a given $\epsilon$ and $x$ we have that $|h(4^N+x)-h(4^N)-h(x)|<\epsilon$
if $N$ is sufficiently large. It follows that \[h\left(4^{n_1}+4^{n_1+n_2}+\ldots+4^{n_1+n_2+\ldots+n_k}\right)\approx \frac k3\] for $n_1<n_2<\cdots<n_k$ with sufficiently large gaps. If $k$ is a multiple of three
and if all $n_i$ are odd, then $x=4^{n_1}+4^{n_1+n_2}+\ldots+4^{n_1+n_2+\ldots+n_k}$ is divisible by three and $a=2x/3$ is an entry in the sequence $a_n$. Since $h(x)$ is equal to the sum in Equation~\ref{eq2}, we see that it can be arbitrarily large.
\end{proof}

Computer experiments suggest that Equation~\ref{conj1} holds for the $(2,1)$-Queen,
but we are unable to prove this. It is possible
to prove that $a_{b_n}-a_n-b_n$ is bounded by
a result of Boris Adamczewski~\cite{A} on balancedness of words. More specifically, his result implies
that if there are $n$ ones in a factor of length $k$
of a Pisot word, then $|a_n-k|$ is bounded.
Since the substitution $a\to aab,\ b\to aa$ is Pisot,
Adamczewski's result implies that $a_{b_n}-b_n-a_n$ is bounded for the $(2,1)$-Queen.

\section{\texttt{Walnut} supported proof of Theorem~2}

Jeffrey Shallit~\cite{S2} showed how to use \texttt{Walnut} to prove that 
the coded Tribonacci word produces the $P$-positions of
the Queen Dee. 
It is straightforward to adapt his proof for the 2-Queen Dee, as follows.
We say that a triple $(n,x,y)$ with $x<y$ is good if $(x,y)$ has no move to $P_{n-1}$.
By Lemma~\ref{lem1} such a triple satisfies the following conditions
\begin{itemize}
\item $\forall\ k<n\ x\not=a_k\ \wedge \ x\not=b_k$
\item $\forall\ k<n\ y\not=a_k\ \wedge \ y\not=b_k$
\item $\forall\ k<n\ |(y-x)-(b_k-a_k)|>1$
\item $\forall\ k<n\ |(y-x)-(b_k+a_k)|>1$
\end{itemize}
These constraints are first order logic statements that can be implemented
in \texttt{Walnut}:
\medbreak

\begin{raggedleft}
\texttt{
def good "?msd\_trib y>x \&\newline 
(\textasciitilde Ek k<n \& \$uaut(k,x)) \&
(\textasciitilde Ek k<n \& \$vaut(k,x)) \&
\newline
(\textasciitilde Ek k<n \& \$uaut(k,y)) \&
(\textasciitilde Ek k<n \& \$vaut(k,y)) \&
\newline
(\textasciitilde Ek,a,b k<n \& \$uaut(k,a) \& \$vaut(k,b) \& y-x=b-a) \&
\newline
(\textasciitilde Ek,a,b k<n \& \$uaut(k,a) \& \$vaut(k,b) \& y-x=b-a+1) \&
\newline
(\textasciitilde Ek,a,b k<n \& \$uaut(k,a) \& \$vaut(k,b) \& y-x+1=b-a) \&
\newline
(\textasciitilde Ek,a,b k<n \& \$uaut(k,a) \& \$vaut(k,b) \& y-x=b+a+1) \&
\newline
(\textasciitilde Ek,a,b k<n \& \$uaut(k,a) \& \$vaut(k,b) \& y-x+1=b+a) \&
\newline
(\textasciitilde Ek,a,b k<n \& \$uaut(k,a) \& \$vaut(k,b) \& y-x=b+a)":}
\end{raggedleft}
\medbreak
The $P$-position $(a_n,b_n)$ is the unique 
good position that has minimal Manhattan distance to
the origin, as specified by Lemma \ref{lem2}.
To finish the proof of Theorem~\ref{thm2}, we 
need to verify that $(n,a_n,b_n)$
satisfies:
\begin{enumerate}
\item The triple $(n,a_n,b_n)$ is good.
\item If $(n,x,y)$ is good then $a_n\leq x$.
\item If $(n,a_n,y)$ is good then $b_n\leq y$.
\end{enumerate}
These are three checks for \texttt{Walnut}:
\newline
\texttt{eval check1 "?msd{\textunderscore}trib An,a,b (n>=1 \& \$uaut(n,a) \& \$vaut(n,b)) => \$good(n,a,b)":}
\newline
\texttt{eval check2 "?msd{\textunderscore}trib An,x,y (n>=1 \& \$good(n,x,y))=>(Ea \$uaut(n,a)\&x>=a)":}
\newline
\texttt{eval check3 "?msd{\textunderscore}trib An,a,y (n>=1 \& \$uaut(n,a) \& \$good(n,a,y)) => (Eb \$uaut(n,b) \& y>=b)":}
\newline \texttt{Walnut} returns \texttt{TRUE} for all three checks.

\section{Acknowledgement}

We would like to thank Karma Dajani for a useful conversation on functions of the type 
$f(x)=\sum_{n=1}^\infty \{ x/\beta^n\}$ for $\beta>1$,
and Neil Sloane for pointing our attention to sequences 
A026367 and A026368.
The anonymous referees made helpful remarks to improve our exposition.
\medbreak
Dan Rust was supported by NWO visitor grant 040.11.700.


\noindent
Institute of Applied Mathematics\\
Delft University of Technology\\
Mourikbroekmanweg 6 \\
2628 XE Delft, The Netherlands\\
\texttt{r.j.fokkink@tudelft.nl}
\\[2mm]
Department of Mathematics\\
Ateneo de Manila University\\
Quezon City 1108, Philippines\\
\texttt{gfmortega@gmail.com}
\\[2mm]
School of Mathematics and Statistics\\ 
The Open University, 
Walton Hall\\ 
Milton Keynes\\ 
MK7 6AA\\ 
UK\\
\texttt{dan.rust@open.ac.uk}

\end{document}